\documentclass[10.5pt, leqno, twoside]{article}   
\usepackage[applemac]{inputenc}
\usepackage{datetime}
\usepackage[T1]{fontenc}
\usepackage{mathtools, bm}
\usepackage{amsthm}
\usepackage{amssymb, bm}
\usepackage{bbm}
\usepackage[mathscr]{euscript}
\usepackage[margin=1in]{geometry}
\usepackage{systeme}
\usepackage{chngcntr}
\usepackage{titling}
\usepackage{cancel}
\usepackage{stmaryrd}
\usepackage{graphicx}
\usepackage{siunitx}
\usepackage{fancyhdr, ifthen}
\usepackage[normalem]{ulem} 
\usepackage{indentfirst}
\usepackage{amsmath}

\newtheoremstyle{theoremnoperiod}
  {\topsep}   
  {\topsep}   
  {\normalfont}  
  {0pt}       
  {\bfseries} 
  {}          
  {5pt plus 1pt minus 1pt} 
  {}          

\newtheorem*{theorem*}{Theorem}
\newtheorem{theorem}{Theorem}[section]

\newtheorem{lemma}[theorem]{Lemma}

\newtheorem*{conjecture*}{Conjecture}

\numberwithin{equation}{section}

\theoremstyle{theoremnoperiod}
\newtheorem*{thmnodot*}{Théorème}

\def\dd{{\rm d}}
\DeclareMathOperator{\N}{\mathbb{N}}
\DeclareMathOperator{\Z}{\mathbb{Z}}
\DeclareMathOperator{\R}{\mathbb{R}}

\DeclareMathOperator{\sA}{\mathscr{A}}

\newcommand{\bsv}{{\boldsymbol v}}
\newcommand{\bsw}{{\boldsymbol w}}
\newcommand{\bsl}{{\boldsymbol \ell}}

\DeclareMathOperator{\e}{\rm e}
\def\d{\,{\rm d}}
\def\1{{\bf 1}}

\renewcommand{\leq}{\leqslant}

\renewcommand{\geq}{\geqslant}

\usepackage{color}
\definecolor{vert}{rgb}{0,0.5,0}
\definecolor{violet}{rgb}{0.7,0.1,0.8}
\definecolor{orange}{rgb}{1,0.3,0}

\graphicspath{ {./images/} }
\providecommand{\keywords}[1]
{
  {\leftskip15mm\small	
  \textbf{Keywords} #1
\par }}
\providecommand{\amsclass}[1]
{
  {\leftskip15mm\small	
  \textbf{2020 Mathematics Subject Classification} #1
\par }}
\begin{document}
\title{On integral boxes of minimal surface}
\author{Jonathan Rotgé \& Gérald Tenenbaum}
\newcommand{\Addresses}{{
  \bigskip
  \footnotesize
  \textsc{Université d'Aix-Marseille, Institut de Mathématiques de Marseille CNRS UMR 7373, 163 avenue de Luminy, Case 907, 13288 Marseille Cedex 9, France}\par \nopagebreak
\textit{E-mail address:} \texttt{jonathan.rotge@etu.univ-amu.fr}\par 
\medskip\medskip
\noindent \textsc{Universit\' e de Lorraine,
B.P. 70239, 
F-54506 Vand\oe{}uvre-l\`es-Nancy Cedex, France}
\par \nopagebreak
\textit{E-mail address:} 
\texttt{Gerald.Tenenbaum@univ-lorraine.fr}
}}
\date{}
\maketitle
\begin{abstract}
Generalising the two-dimensional case, we provide estimates for the mean-values of the lengths of the edges of an integral box with given volume and minimal surface.
\end{abstract}
\keywords{Localised divisors, averages of arithmetic functions}
\amsclass{11N25, 11N37}
\section{Introduction and statements of results}

Given an integer $k\geqslant 2$, consider for each integer $n$ a $k$-dimensional box with integral edges and volume~$n$. How should we select the lengths $d_1,\ldots,d_k$ of the edges so that the surface of the box is minimal? Since the surface of a $k$-dimensional box with edges $d_1,\dots,d_k$ is given by
\[\sigma(d_1,\dots,d_k):=2\sum_{1\leq h\leq k}\prod_{\substack{1\leq m\leq k\\ m\neq h}}d_m\quad (k\geq 2),\]
it is clear that any sequence $\{\varrho_{h}(n)\}_{1\leq h\leq k}=\{\varrho_{k,h}(n)\}_{1\leq h\leq k}$ realizing this minimum and arranged in increasing order is a solution of the optimisation problem
\begin{equation*}
(P_{n,k})\quad
	\begin{cases}& \displaystyle\prod_{1\leq h\leq k}\varrho_{h}(n)=n,\\
	& \varrho_{1}(n)\leq\varrho_{2}(n)\leq\cdots\leq\varrho_{k}(n),\\
	& \displaystyle\sum_{1\leq h\leq k}\frac1{\varrho_{h}(n)}=\min_{d_1\dots d_k=n}\sum_{1\leq h\leq k}\frac1{d_h}\cdot
	\end{cases}
\end{equation*}
\par
This setting generalizes the case $k=2$ introduced by the second author in \cite{tenenbaum_76}.  In this case we have
\begin{equation}\label{eq:defrho21rho22}\varrho_{1}(n)=\max\{d|n:d\leq\sqrt n\},\quad \varrho_{2}(n)=\min\{d|n:d\geq\sqrt n\},\end{equation}
so  the pair $(\varrho_1(n),\varrho_2(n))$ is unique. We note in passing that uniqueness of the solution to  problem $(P_{n,k})$ is not necessarily granted in the general case. We leave this question open for the time being and merely indicate that numerical tests show that uniqueness holds for $k\in\{3,4,5\}$ and $n\leq 10^8$.
\par 
Estimates for the mean-values of $\varrho_{2,1}(n)$ and $\varrho_{2,2}(n)$ were given in \cite{tenenbaum_76}. The purpose of this note is to provide corresponding estimates in the general case of dimension $k$. This last parameter will be fixed throughout the paper,  consequently we shall omit to indicate dependence upon it. \par 
We shall see that evaluating the average of $\varrho_j$ is significantly more delicate for $j=1$ than for $j\geqslant 2$. By \cite[th.~2]{tenenbaum_76}, we have
\begin{equation}\label{eq:valmoy:rho22:ten}
	\sum_{n\leq x}\varrho_{2,2}(n)=\frac{\pi^2 x^2}{12\log x}\bigg\{1+O\bigg(\frac{1}{\log x}\bigg)\bigg\}\quad (x\geq 2).
\end{equation}
Improving upon estimates established in \cite[th.~1]{tenenbaum_76} and (implicitly) in \cite{GT84}, Ford \cite[cor.~6]{ford_2008} showed that
\begin{equation}\label{th:ford_res_'08}
	\sum_{n\leq x}\varrho_{2,1}(n)\asymp\frac{x^{3/2}}{(\log x)^{\delta}(\log_2 x)^{3/2}}\quad (x\geq 3),
\end{equation}
with 
\begin{equation}\label{def:alpha}
	\delta:=1-\frac{1+\log_2 2}{\log 2}\approx0.086071.
\end{equation}
 Here and throughout, we write $\log_j$ for the $j$th iterated logarithm. \par \goodbreak
De Koninck and Razafindrasoanaivolala (\cite{dekon_razaf_2020},\cite{dekon_razaf_2023}) estimated  the mean-values  of the ratios $\varrho_{2,1}(n)/\varrho_{2,2}(n)$ and $\varrho_{2,2}(n)/\varrho_{2,1}(n)^r$, $r>-1$, and also of ratios involving the logarithms of these functions.\par
A very recent result of Haddad \cite{TH25} provides the formula
$$\sum_{n\leqslant x}\log \varrho_{2,1}(n)=cx\log x+O(x)$$
for a suitable constant $c\in]0,\frac12]$.
\goodbreak
\par\medskip
We state our results as the two  following theorems.  We denote  Riemann's zeta function by $\zeta(s)$. We also put
\begin{equation}\label{def:Q;alphak}
	Q(v):=v\log v-v+1\quad (v>0),\quad \delta_k:=Q\Big(\frac{k-1}{\log k}\Big)\quad (k\geqslant 2).
\end{equation}

\begin{theorem}\label{th:avg_value_smallest_side}
	If the sequence $\{\varrho_h(n)\}_{h=1}^{k}$ satisfies $(P_{n,k})$, we have
	\begin{equation}\label{eq:asymp:valmoy:rhok1}
		\sum_{n\leq x}\varrho_{1}(n)\asymp_k\frac{x^{1+1/k}}{(\log x)^{\delta_k}(\log_2 x)^{3/2}}\quad (x\geq 3).
	\end{equation}
\end{theorem}
\begin{theorem}\label{th:avg_value_largest_side}
	Let $2\leq j\leq k$ and $\gamma_j:=k+1-j$. If the sequence $\{\varrho_h(n)\}_{h=1}^{k}$ satisfies $(P_{n,k})$, we have
	\begin{equation}
	\label{vmrhoj}\sum_{n\leq x}\varrho_{j}(n)=\frac{x^{1+1/\gamma_{j}}}{(\log x)^{\gamma_{j}}}\bigg\{\frac{\gamma_{j}^{2\gamma_{j}}}{(\gamma_{j}+1)!}\zeta\bigg(1+\frac1{\gamma_{j}}\bigg)+O\bigg(\frac{1}{\log x}\bigg)\bigg\}\quad (x\geq 2).
	\end{equation}
\end{theorem}
\noindent{\it Remarks.} (i) The upper bound in \eqref{eq:asymp:valmoy:rhok1} is actually valid if $\varrho_1(n)$ is replaced by the smallest divisor $d_1$ in any representation $n=d_1\cdots d_k$.
\par (ii) The trivial bound $\varrho_{j}(n)^{k+1-j}\leq n$ readily implies
\begin{equation*}\label{def:avgord:rhokj}
	\sum_{n\leq x}\varrho_{j}(n)\leq\frac{2x^{1+1/\gamma_j}}{1+\gamma_j}\quad (1\leq j\leq k).
\end{equation*}
\par 
(iii) Generalising an observation made in \cite{tenenbaum_76} for the case $k=2$, we shall see that the average of $\varrho_j(n)$ is dominated by integers with exactly $k+1-j$ “large” prime factors, in a sense to be made precise later. 
\par
(iv) For $k=j=2$, \eqref{vmrhoj} coincides with \eqref{eq:valmoy:rho22:ten}.
\par 
(v) For $2\leq j\leq k$ and $0\leq h\leq j-2$, we have
\[\sum_{n\leq x}\varrho_{k-h,j-h}(n)\sim\sum_{n\leq x}\varrho_{k,j}(n)\qquad (x\to\infty).\]
\par

\section{A lemma}

For any ${\bsv}=(v_1,\dots,v_k)\in(\N^*)^k$, set
\[S({\bsv}):=\sum_{1\leq j\leq k}\frac1{v_j}
\cdot
\]
and let us equip
\[E_n:=\{\bsv\in(\N^*)^k:v_1\leq\cdots\leq v_k,\,v_1v_2\cdots v_k=n\},\]
with the total preorder relation $\preccurlyeq$ defined by
\[\bsv\preccurlyeq\bsw\Leftrightarrow S(\bsv)\leqslant S(\bsw)\quad (\bsv,\bsw\in E_{n}).\]
\par\medskip

The following result provides a necessary condition for optimality of a $k$-tuple. We denote by $P^-(n)$ the smallest prime factor of an integer $n>1$ and make the standard convention $P^-(1)=+\infty$.

\begin{lemma}\label{l:condmeilleuruplet}
	Let $n\geq 1$ and $\bsv\in E_n$. If $v_jP^-(v_h)<v_h$ for some $1\leq j<h\leq k$, then the ordered $k$-tuple $\bsw$ obtained from $\bsv$ on replacing $v_j$ by $v_jP^-(v_h)$ and  $v_h$ by $v_h/P^-(v_h)$ is an element of $E_n$ satisfying $\bsw\prec\bsv$.
\end{lemma}
\begin{proof}
	Since $\bsv\in E_n$ and $\bsw$ is ordered, we have  $\bsw\in E_n$. Moreover, 
		\[S(\bsw)-S(\bsv)=\frac{\{P^-(v_h)-1\}\{v_jP^-(v_h)-v_h\}}{v_jv_hP^-(v_h)}< 0.\qedhere\]
\end{proof}
For purpose of further reference, we note as an immediate consequence of Lemma~\ref{l:condmeilleuruplet} that, for each $n\geq 1$, we  have
\begin{equation}\label{prop:rhokj}
	\varrho_{j}(n)P^-(\varrho_{h}(n))\geq\varrho_{h}(n)\quad (1\leq j<h\leq k).
\end{equation}

\section{Proof of Theorem \ref{th:avg_value_smallest_side}}

Set
\[N_{j,\ell}(x):=\bigg\{n\leq x:\frac{x^{1/k}}{2^{\ell+1}}<\varrho_{j}(n)\leq\frac{x^{1/k}}{2^{\ell}}\bigg\}\quad (x\geq 1,\, 1\leq j\leq k,\,\ell\in\Z).\]
	
\begin{lemma}\label{l:encad:rhokj:rhok1}
	Let $x\geq 1$. For each $n\leq x$, there exists a unique $\ell\geq0$ such that $n\in N_{1,\ell}(x)$ and
	\begin{equation}\label{eq:boundings:rhokj}
		\frac{x^{1/k}}{2^{\ell+1}}<\varrho_{j}(n)\leq2^{(\ell+1)(k-1)}x^{1/k}\quad (2\leq j\leq k).
	\end{equation}
\end{lemma}
\begin{proof}  The requirement that $n\in N_{1,\ell}(x)$ implies that $\ell$ is the greatest integer such that $\varrho_1(n)\leqslant x^{1/k}/2^\ell$. Moreover, since $\varrho_1(n)\leqslant x^{1/k}$ by definition, we must have $\ell\geqslant 0$.
	\par 
	When $2\leqslant j\leqslant k$, we have $\varrho_{j}(n)\geq\varrho_{1}(n)>x^{1/k}/2^{\ell+1}$ and
\[\varrho_{j}(n)\leq\varrho_{k}(n)\leq\frac{x}{\{\varrho_{1}(n)\}^{k-1}}\leq 2^{(\ell+1)(k-1)}x^{1/k}.\qedhere\]
\end{proof}

We can now embark on the proof of Theorem \ref{th:avg_value_smallest_side}. 
	For $\bsv=(v_1,\dots,v_m)\in[0,\infty[^m\ (m\geq 1)$, define
	\begin{gather*}
		\tau_{m+1}(n,\bsv):=|\{(d_1,\dots,d_m)\in\N^m:d_1\dots d_m\mid n,\, v_j<d_j\leq 2v_j\ (1\leq j\leq m)\}|\quad (n\geq 1),\\
		H^{(m+1)}(x,\bsv):=|\{n\leq x:\tau_{m+1}(n,\bsv)\geq 1\}|\quad (x\geq 1,\,m\geq 1).
	\end{gather*}
	We know from \cite[th.~1]{koukoulopoulos_2013} that
	\begin{equation}
	\label{estkoukou}
	H^{(m+1)}(x,\bsv)\asymp\frac{x}{(\log x)^{Q(1/\log r)}(\log_2x)^{3/2}},
	\end{equation}
	with $r:=(m+1)^{1/m}$, provided  \mbox{$2^{m+1}\prod_{1\leqslant h\leqslant m}v_h\leqslant x/v_1^{c_1}$} and $v_m\leqslant v_1^{c_2}$ for suitable constants $c_1$, $c_2$, \mbox{$0<c_1\leqslant 1\leqslant c_2$}. Note right-away that $Q(1/\log r)=\delta_k$ as defined in \eqref{def:Q;alphak} when $m=k-1$. 
	\par 
	By Lemma~\ref{l:encad:rhokj:rhok1}, for every $n\leq x$, there exists a unique $k$-tuple $\bsl_n=(\ell_{n,1},\dots,\ell_{n,k})\in\Z^{k}$ such that $n\in N_{j,\ell_{n,j}}(x)$ $(1\leq j\leq k)$, $\ell_{n,1}\geqslant 0$, and $-(\ell_{n,1}+1)(k-1)\leq\ell_{n,k}\leq \cdots \leq\ell_{n,1}$. Therefore
	\begin{equation}\label{eq:majo:Nk1lx}
		|N_{1,\ell}(x)|\leq \sum_{\substack{(\ell+1)(1-k)\leq\ell_k\leq\cdots\leq\ell_2\leq\ell\\ \ell+\ell_2+\cdots+\ell_k\geq 0}}H^{(k)}\Big(\frac{x}{2^{\ell+\ell_2+\cdots+\ell_k}},\Big\{\frac{x^{1/k}}{2^{\ell_j+1}}\Big\}_{2\leq j\leq k}\Big)\quad (\ell\geq 0).
	\end{equation}
	Now, setting $L_x:=\lceil2\delta_k(\log_2 x)/\log 2\rceil\ (x\geq 3)$,  we can write
	\begin{align*}
		\sum_{n\leq x}\varrho_{1}(n)&=\sum_{\substack{n\leq x\\ \varrho_{1}(n)\geq x^{1/k}/(\log x)^{2\delta_k}}}\varrho_{1}(n)+O\bigg(\frac{x^{1+1/k}}{(\log x)^{2\delta_k}}\bigg)\leq x^{1/k}\sum_{0\leq\ell\leq L_x}\frac{|N_{1,\ell}(x)|}{2^\ell}+O\bigg(\frac{x^{1+1/k}}{(\log x)^{2\delta_k}}\bigg).
	\end{align*}
	Since the required hypotheses for \eqref{estkoukou} are satisfied for $v_j:=x^{1/k}/2^{\ell_j+1}$ $(2\leqslant j\leqslant k)$, $0\leqslant \ell\leqslant L_x$, $\ell_j\asymp\ell+1$, and sufficiently large $x$, we deduce from \eqref{eq:majo:Nk1lx} that
	\begin{align*}
		\sum_{n\leq x}\varrho_{1}(n)&\ll_k\frac{x^{1+1/k}}{(\log x)^{\delta_k}(\log_2 x)^{3/2}} \sum_{0\leq\ell\leq L_x}\frac1{2^{\ell}}\sum_{\substack{(\ell+1)(1-k)\leq\ell_k\leq\cdots\leq\ell_2\leq\ell\\ \ell+\ell_2+\cdots+\ell_k\geq 0}}\frac{1}{2^{\ell+\ell_2+\cdots+\ell_k}}\\
		&\ll_k\frac{x^{1+1/k}}{(\log x)^{\delta_k}(\log_2 x)^{3/2}}\sum_{0\leq\ell\leq L_x}\frac{\{k(\ell+1)\}^{k-1}}{2^{\ell}}\ll_k\frac{x^{1+1/k}}{(\log x)^{\delta_k}(\log_2 x)^{3/2}}\cdot
	\end{align*}
	\par
	It remains to establish the lower bound included in \eqref{eq:asymp:valmoy:rhok1}. Set
	\[\sA(x):=\big|\big\{\tfrac12x<n\leq x:\mu(n)^2=1,\,\tau_{k}\big(n,\big\{\tfrac12x^{1/k}\big\}_{1\leq j\leq k-1}\big)\geq 1\big\}\big|\quad (x\geq 1),\]
	where, here and throughout, $\mu$ refers to the Möbius function. Observe that
	\[\frac1{\varrho_{1}(n)}< \sum_{1\leq j\leq k}\frac1{\varrho_{j}(n)}<\frac{2k}{x^{1/k}}\qquad \big(n\in\sA(x)\big),\]
	hence $x^{1/k}/2k<\varrho_{1}(n)\leq x^{1/k}$ by construction. By  \cite[th.~2]{koukoulopoulos_2013}, it follows that
	\begin{align*}
		\sum_{n\leq x}\varrho_{1}(n)&\geq\sum_{\substack{n\leq x\\ x^{1/k}/2k<\varrho_{1}(n)\leq x^{1/k}}}\varrho_{1}(n)\geq\frac{x^{1/k}}{2k}\Big|\Big\{n\leq x:\frac{x^{1/k}}{2k}<\varrho_{1}(n)\leq x^{1/k}\Big\}\Big|\\
		&\geq\frac{x^{1/k}\sA(x)}{2k}\gg_k \frac{x^{1+1/k}}{(\log x)^{\delta_k}(\log_2 x)^{3/2}}\cdot
	\end{align*}
	 This completes the proof.

\section{Proof of Theorem \ref{th:avg_value_largest_side}}

Recall definition  $\gamma_{j}:=k+1-j$, and put
\begin{equation}\label{def:alphakj}
	\alpha_{j}:=\frac{1}{\gamma_{j}+1/2}=\frac{1}{k-j+3/2}\quad (1\leq j\leq k).
\end{equation}

Our proof of Theorem \ref{th:avg_value_largest_side} is based on the observation  that the sums \eqref{vmrhoj} are dominated by large values of $\varrho_{j}(n)$. More precisely, we shall see that the structure of the integers $n\leq x$ such that $\varrho_{j}(n)>x^{\alpha_{j}}$ is very constrained: in that case the $\varrho_{h}(n)\ (j\leq h\leq k)$ are all prime.  To lighten  notation, we write
\begin{equation}
	\pi_{j}(n):=\prod_{1\leq h\leq j}\varrho_{h}(n)\quad (n\geq 1,\,1\leq j\leq k).
\end{equation}

\begin{lemma}\label{l:rholtouspremiers}
	Let $1\leqslant n\leqslant x$ and $2\leq j\leq k$. If $\varrho_{j}(n)>x^{\alpha_{j}}$, then $\varrho_{h}(n)$ is a prime number for $j\leqslant h\leqslant k$.
\end{lemma}
\begin{proof}
	Under the assumption $\varrho_{j}(n)>x^{\alpha_{j}}$, we have
	\begin{equation}\label{eq:majoration:pikj-1}
		\pi_{j-1}(n)\leq\frac{x}{\{\varrho_{j}(n)\}^{\gamma_{j}}}<x^{1-\gamma_{j}\alpha_{j}}=x^{1/(2\gamma_{j}+1)}.
	\end{equation}
	If $\varrho_{h}(n)$ is composite for some $h\in[j,k]$, then $\varrho_{h}(n)\geq p^2$, with $p:=P^-(\varrho_{h}(n))$. However, in view of Lemma~\ref{l:condmeilleuruplet} we have
	\begin{equation}\label{eq:mino:P-rhokl}
		p\geq\frac{\varrho_{h}(n)}{\varrho_{j-1}(n)}\geq\frac{\varrho_{j}(n)}{\varrho_{j-1}(n)}\geq\frac{\varrho_{j}(n)}{\pi_{j-1}(n)}>x^{1/(2\gamma_{j}+1)},
	\end{equation}
	by \eqref{eq:majoration:pikj-1}, while, since $\varrho_{j-1}(n)p\geq \varrho_{h}(n)\geq p^2$, we also have
	\begin{equation}\label{eq:majo:P-rhokl}
		p\leq\varrho_{j-1}(n)\leq\pi_{j-1}(n)<x^{1/(2\gamma_{j}+1)},
	\end{equation}
	by another appeal to  \eqref{eq:majoration:pikj-1}. Since \eqref{eq:mino:P-rhokl} and \eqref{eq:majo:P-rhokl} are incompatible, $\varrho_{h}(n)$ must be prime. 
\end{proof}
\medskip
Observe that, trivially, $\varrho_j(n)\leqslant x^{1/\gamma_j}$ $(1\leqslant j\leqslant k)$ and that the contribution to the left-hand side of~\eqref{vmrhoj} of those integers $n$ such that $\varrho_j(n)\leqslant x^{\alpha_j}$ is clearly  $\ll x^{1+\alpha_j}$, a quantity exceeded by the error term of~\eqref{vmrhoj}. Therefore, we can focus on those $n$ such that $\varrho_{j}(n)\in]x^{\alpha_{j}},x^{1/\gamma_j}]$. By Lemma~\ref{l:rholtouspremiers}, these integers admit a representation of the form $n=mpq_1\dots q_{k-j}$ with $p\leq q_1\leq\cdots\leq q_{k-j}$, where $p$ and the $q_h$ are all prime. It follows that
\begin{equation}
\label{decsommerhoj}
\sum_{n\leqslant x}\varrho_j(n)=S_j(x)+O\big(x^{1+\alpha_j}\big),
\end{equation}
with 
\begin{equation}\label{def:Tk}
	S_{j}(x):=\sum_{x^{\alpha_{j}}<p\leq x^{1/\gamma_j}}p\sum_{\substack{p\leq q_1\leq \dots\leq q_{k-j}\\ q_1\dots q_{k-j}\leq x/p}}\bigg\lfloor\frac{x}{pq_1\dots q_{k-j}}\bigg\rfloor\quad (x\geq 1,\,2\leqslant j\leqslant k).
\end{equation}
\par \goodbreak
We now evaluate $S_j(x)$.  Let $\Omega(n)$ denote the total number of prime factors of an integer $n$. We have
 \begin{equation}
 \begin{aligned}
 \label{defTj}
 S_j(x)&=\sum_{x^{\alpha_{j}}<p\leq x^{1/\gamma_j}}p\sum_{\substack{n\leqslant x/p\\P^-(n)\geqslant p\\\Omega(n)=k-j}}\sum_{m\leqslant x/np}1=\sum_{m\leqslant x^{1-\gamma_j\alpha_j}}\sum_{x^{\alpha_j}<p\leqslant (x/m)^{1/\gamma_j}}p\sum_{\substack{n\leqslant x/pm\\P^-(n)\geqslant p\\\Omega(n)=k-j}}1,
 \end{aligned}
 \end{equation}
\par 
Now from the trivial estimate
\[\sum_{p\leqslant x^{\alpha_j}}p\sum_{\substack{n\leqslant x/pm\\ P^-(n)\geqslant p\\ \Omega(n)=k-j}}1\leqslant \sum_{p\leqslant x^{\alpha_j}}p\sum_{\substack{n\leqslant x/pm\\ P^-(n)\geqslant p\\\Omega(n)=k-j}}\frac{x}{pmn}\ll  \frac{x\pi(x^{\alpha_j})(\log_2x)^{k-j}}{m}\ll\frac{x^{1+\alpha_j}(\log_2x)^{k-j}}{m\log x}\]
we see that removing the lower bound for $p$  in the summation conditions of \eqref{defTj} introduces a global error $\ll x^{1+\alpha_j}(\log_2x)^{k-j}$.
Observing that $1-\gamma_j\alpha_j=\alpha_j/2$ and defining
\begin{equation}
  T_j(y):=\sum_{p\leqslant y^{1/\gamma_j}}p\sum_{\substack{n\leqslant y/p\\ P^-(n)\geqslant p\\ \Omega(n)=k-j}}1\leqslant \sum_{p\leqslant y^{1/\gamma_j}}\sum_{\substack{n\leqslant y/p\\ P^-(n)\geqslant p\\\Omega(n)=k-j}}\frac{y}{n}\ll y^{1+1/\gamma_j} ,
  \label{defUj}
\end{equation}
we arrive at
\begin{equation}
\label{estSj}
S_j(x)=\sum_{m\leqslant x^{\alpha_j/3}}T_j\Big(\frac xm\Big)+O\Big(x^{1+(1-\alpha_j/3)/\gamma_j}\Big),
\end{equation}
in view of the inequality $\alpha_j<(1-\alpha_j/3)/\gamma_j$.

\smallskip
The next lemma furnishes an asymptotic formula for $T_{j}(y)$ as defined in \eqref{defUj}. 
\begin{lemma}
	Let $2\leq j\leq k$. We have 
	\begin{equation}\label{eq:est:Tk}
		T_{j}(y)=\frac{\gamma_{j}^{2\gamma_{j}}y^{1+1/\gamma_j}}{(\gamma_{j}+1)!(\log y)^{\gamma_{j}}}\bigg\{1+O\bigg(\frac{1}{\log y}\bigg)\bigg\}\quad (y\geq 2).
	\end{equation}
\end{lemma}
\begin{proof}
	Set
	\[A_{h}(v,t):=\sum_{\substack{q_1\geqslant \dots\geqslant q_{h}\geqslant t\\ q_1\dots q_h\leq v}}1\quad (h\geq 1,\,2\leq t\leq v),\]
	so that
	\begin{equation}\label{rep:Uj;A_j}
		T_{j}(y)=\int_{2}^{y^{1/\gamma_j}}tA_{k-j}\Big(\frac{y}{t},t\Big)\d\pi(t).
	\end{equation}
	 An iterated application of a strong form of the prime number theorem yields, for any fixed $h\geqslant 1$,
	 	\[A_{h}(v,t)=\int_{\substack{u_1\geqslant \dots\geqslant u_h\geq t\\ u_1\cdots u_h\leq v}}\d\pi(u_1)\cdots\d\pi(u_h)=\int_{\substack{u_1\geqslant \cdots\geqslant u_h\geq t\\ u_1\cdots u_h\leq v}}\frac{\d u_{1}\cdots \d u_{h}}{(\log u_1)\cdots(\log u_h)}+O\Big(v\e^{-\sqrt{\log v}}\Big).\]
	\goodbreak
	\noindent Therefore
	\[T_{j}(y)={1\over (k-j)!}\int_{2}^{y^{1/\gamma_j}}\int_{\substack{u_1,\dots,u_{k-j}\geq t\\ u_1\cdots u_{k-j}\leq y/t}}\frac{t\d u_{1}\cdots \dd u_{k-j}\dd t}{(\log t)(\log u_1)\dots(\log u_{k-j})}+O\Big(y^{1+1/\gamma_j}\e^{-\sqrt{\log y}}\Big).\]
Performing the change of variables $t=y^{v_0}$, $u_h=y^{v_h}\ (1\leq h\leq k-j)$, we obtain
	\[T_{j}(y)={1\over (k-j)!}\int_{(\log 2)/\log y}^{1/\gamma_j}\int_{\substack{v_1,\dots,v_{k-j}\geq v_0\\ v_1+\cdots+v_{k-j}\leq 1-v_0}}\frac{y^{2v_0+v_1+\cdots+v_{k-j}}\d v_0\d v_1\cdots \dd v_{k-j}}{v_0v_1\cdots v_{k-j}}+O\Big(y^{1+1/\gamma_j}\e^{-\sqrt{\log y}}\Big).\]
	For $2\leq h\leq k-j$, write
	\begin{gather*}
		 W_h:=\sum_{1\leq m\leq h}w_{m}\quad ( w_1,\dots,w_{h}\in\R).
	\end{gather*}
	The further change of variables
	\[v_0=\frac1{\gamma_j}-\frac{s}{\log y},\quad v_h=v_0+\frac{w_h}{\log y}=\frac1{\gamma_j}+\frac{w_h-s}{\log y}\quad (1\leq h\leq k-j)\]
	yields
	 \begin{equation}\label{eq:decomp:Tk;1;2}
		\begin{aligned}
			T_{j}(y)&={U_j(y)\over (k-j)!}+O\bigg(\frac{y^{1+1/\gamma_j}}{\e^{\sqrt{\log y}}}\bigg)\\
		\end{aligned}
	\end{equation}
	with
	\begin{equation*}\begin{aligned}
	U_j(y)&:=\frac{y^{1+1/\gamma_j}}{(\log y)^{\gamma_{j}}}\int_0^{\log(y^{1/\gamma_j}/2)}\frac{\dd s}{\e^{(\gamma_{j}+1)s}}\int_{\substack{w_1,\dots,w_{k-j}\geq 0\\ W_{k-j}\leq \gamma_{j}s}}\frac{{\e^{W_{k-j}}}\d w_1\cdots \d w_{k-j}}{D_{j}(s,y)},\\ D_{j}(s,y)&:=\Big(\frac1{\gamma_j}-\frac{s}{\log y}\Big)\prod_{1\leq h\leq k-j}\Big(\frac1{\gamma_j}+\frac{w_h-s}{\log y}\Big).\end{aligned}
	\end{equation*}
	For $ s\leq \log(\tfrac12y^{1/\gamma_j})$, we have
	\[\frac1{\gamma_j}+\frac{w_h-s}{\log y}\geq \frac1{\gamma_j}-\frac{s}{\log y}\geq\frac{\log 2}{\log y},\]
	hence the contribution of $s\geqslant Y:=(\gamma_j+2)\log_2y$ to $U_j(y)$ is
	\begin{equation}\label{eq:est:majo:Tk2}
		\begin{aligned}
			&\ll y^{1+1/\gamma_j}\int_{Y}^{\infty}\e^{-s}\d s\int_{\substack{w_1,\dots,w_{k-j}\geq 0\\ W_{k-j}\leq \gamma_{j}s}}\d w_1\cdots\d w_{k-j}\\
			&\ll y^{1+1/\gamma_j}\int_{Y}^{\infty}s^{k-j}\e^{-s}\d s\ll y^{1+1/\gamma_j}Y^{k-j}\e^{-Y}\ll\frac{y^{1+1/\gamma_j}}{(\log y)^{\gamma_{j}+1}}\cdot
		\end{aligned}
	\end{equation}
	Moreover, for $s\leq Y$, we have
	\[D_{j}(s,y)=\gamma_j^{-\gamma_{j}}+\gamma_j^{1-\gamma_{j}}\frac{W_{k-j}-\gamma_{j}s+O(1)}{\log y},\]
	and so
	\[\frac1{D_{j}(s,y)}=\gamma_{j}^{\gamma_{j}}\bigg\{1-\gamma_j\frac{W_{k-j}-\gamma_{j}s+O(1)}{\log y}\bigg\}.\]
	It follows that
	\begin{align*}
		U_j(y)&=\frac{\gamma_{j}^{\gamma_{j}}y^{1+1/\gamma_j}}{(\log y)^{\gamma_{j}}}\int_0^{Y}\frac{\dd s}{\e^{(\gamma_j+1)s}}\int_{\substack{w_1,\dots,w_{k-j}\geq 0\\ W_{k-j}\leq \gamma_{j}s}}\e^{W_{k-j}}\Big\{1+O\Big(\frac{s+1}{\log y}\Big)\Big\}\d w_1\cdots\d w_{k-1}.\\
	\end{align*}
	Since the contribution of the error term is clearly
	\[\ll \frac{y^{1+1/\gamma_j}}{(\log y)^{\gamma_{j}+1}}\int_0^\infty (s+1)^{\gamma_{j}}\e^{-s}\d s\ll\frac{y^{1+1/\gamma_j}}{(\log y)^{\gamma_{j}+1}},\]
	we finally get
	\begin{align}
	\label{Ujinterm}
		U_j(y)&=\frac{\gamma_{j}^{\gamma_{j}}y^{1+1/\gamma_j}}{(\log y)^{\gamma_{j}}}\int_0^{Y}\frac{\dd s}{\e^{(\gamma_j+1)s}}
		\int_{\substack{w_1,\cdots w_{k-j}\geq 0\\ W_{k-j}\leq \gamma_{j}s}}\e^{W_{k-j}}\d w_1\cdots\d w_{k-j}+O\bigg(\frac{y^{1+1/\gamma_j}}{(\log y)^{\gamma_{j}+1}}\bigg).
	\end{align}
	The range of the outer integral may now be extended to $[0,\infty[$ for the involved error is
	$$\ll \int_Y^\infty s^{k-j}\e^{-s}\d s\ll Y^{k-j}\e^{-Y}\ll{1\over (\log y)^{\gamma_j+1}}\cdot$$
	 Inverting the order of integrations in the multiple integral thus modified, we get that it is
	\begin{equation}\label{eq:est:Tk1}
		\begin{aligned}
			\int_{[0,\infty[^{k-j}}{\e^{W_{k-j}}}&\d w_1\cdots\dd w_{k-j}\int_{W_{k-j}/\gamma_{j}}^\infty\frac{\dd s}{\e^{(\gamma_j+1)s}}=\frac{1}{\gamma_{j}+1}\int_{[0,\infty[^{k-j}}\frac{\dd w_1\cdots\dd w_{k-j}}{\e^{W_{k-j}/\gamma_{j}}}
			=\frac{\gamma_{j}^{\gamma_{j}-1}}{\gamma_{j}+1}\cdot
		\end{aligned}
	\end{equation}
	Formula \eqref{eq:est:Tk} then follows from \eqref{eq:decomp:Tk;1;2} and \eqref{Ujinterm}.
\end{proof}

We are now in a position to complete the proof of Theorem \ref{th:avg_value_largest_side}. By \eqref{decsommerhoj}, \eqref{estSj} and \eqref{eq:est:Tk}, we have
\begin{align*}
	\sum_{n\leq x}\varrho_{j}(n)&=\frac{\gamma_{j}^{2\gamma_{j}}x^{1+1/\gamma_j}}{(\gamma_{j}+1)!(\log x)^{\gamma_{j}}}\sum_{1\leq m\leq x^{\alpha_j/3}}\frac{1}{m^{1+1/\gamma_j}}\bigg\{1+O\bigg(\frac{1+\log m}{\log x}\bigg)\bigg\}+O\bigg(\frac{x^{1+1/\gamma_j}}{(\log x)^{\gamma_{j}+1}}\bigg)\\
	&=\frac{\gamma_{j}^{2\gamma_{j}}\zeta(1+1/\gamma_j)x^{1+1/\gamma_j}}{(\gamma_{j}+1)!(\log x)^{\gamma_{j}}}+O\bigg(\frac{x^{1+1/\gamma_j}}{(\log x)^{\gamma_{j}+1}}\bigg),
\end{align*}
as required.

\Addresses

\end{document}